\documentclass[12pt]{article}
\usepackage{amsmath}
\usepackage{amsfonts}
\usepackage{amsthm}
\usepackage{amssymb}
\usepackage{color}
\usepackage{enumerate}

\makeatletter \@addtoreset{equation}{section}

\makeatother

\title{ Caccioppoli--type estimates and Hardy--type inequalities derived from degenerated $p$--harmonic problems }

\usepackage{authblk}

\author[1]{Pavel Dr\'ab{e}k~\thanks{e--mail: pdrabek@kma.zcu.cz}}
\author[2,3]{Agnieszka Ka\l{}amajska~\thanks{e--mail: a.kalamajska@mimuw.edu.pl}}
\author[3]{Iwona Skrzypczak~\thanks{e--mail: iskrzypczak@mimuw.edu.pl}}

\affil[1]{\small
Department of Mathematics and NTIS, University of West Bohemia, Pilsen, Czech Republic }

\affil[2]{\small
 Institute of Mathematics,
 Polish Academy of Sciences at Warsaw, Warsaw, Poland\footnote{temporary address}}

\affil[3]{\small
Faculty of Mathematics, Informatics and Mechanics,
University of Warsaw, Warsaw, Poland }
\date{}

\begin{document}
\maketitle \sloppy

\thispagestyle{empty}

\renewcommand{\it}{\sl}
\renewcommand{\em}{\sl}

\belowdisplayskip=18pt plus 6pt minus 12pt \abovedisplayskip=18pt
plus 6pt minus 12pt
\parskip 4pt plus 1pt
\parindent 0pt

\newcommand{\barint}{
         \rule[.036in]{.12in}{.009in}\kern-.16in
          \displaystyle\int  }
\def\r{{\mathbb{R}}}
\def\n{{\mathbb{N}}}
\def\rn{{\mathbb{R}^{n}}}
\def\rom{{\mathbb{R}^{m}}}
\def\w{\widetilde}
\def\M{{\cal M}}
\def\zi{[0,\infty )}
\def\sf{{\rm supp\,\phi}  }
\def\sx{{\rm supp\,\xi}  }
\def\rp{{\mathbb{R}_{+}}}
\def\bA{{\bar{A}}}
\def\O{{\Omega}}
\def\o{{\omega}}
\def\Oc{{\Omega\cap}}
\def\a{{\alpha}}
\def\b{{\beta}}
\def\e{{\epsilon}}
\def\d{{\delta}}
\def\D{{\Delta}}
\def\t{{\tau}}
\def\s{{\sigma}}
\def\g{{\gamma}}
\def\k{{\kappa}}
\def\na{{\nabla}}
\def\irn{{\int_\rn}}
\def\dsp{\def\baselinestretch{1.37}\large
\normalsize}
\def\dist{{\rm dist}}

\newtheorem{theo}{\bf Theorem}[section]
\newtheorem{coro}{\bf Corollary}[section]
\newtheorem{lem}{\bf Lemma}[section]
\newtheorem{rem}{\bf Remark}[section]
\newtheorem{defi}{\bf Definition}[section]
\newtheorem{ex}{\bf Example}[section]
\newtheorem{fact}{\bf Fact}[section]
\newtheorem{prop}{\bf Proposition}[section]

\newcommand{\ds}{\displaystyle}
\newcommand{\ts}{\textstyle}
\newcommand{\ol}{\overline}
\newcommand{\wt}{\widetilde}
\newcommand{\ck}{{\cal K}}
\newcommand{\ve}{\varepsilon}
\newcommand{\vp}{\varphi}
\newcommand{\pa}{\partial}
\newcommand{\psa}{\Phi_{s,r}}


\parindent 1em

\maketitle

\begin{abstract} We obtain Caccioppoli--type estimates for nontrivial and nonnegative solutions to the anticoercive partial differential inequalities of elliptic type involving degenerated $p$--Laplacian: $-\Delta_{p,a} u:= -\mathrm{div}(a(x)|\na u|^{p-2}\na u)\ge b(x)\Phi (u)$, where $u$ is defined in a domain $\Omega$. Using Caccioppoli--type estimates, we obtain several variants of Hardy--type inequalities in weighted Sobolev spaces.

\end{abstract}
\smallskip

  {\small {\bf Key words and phrases:}  $p$--harmonic  PDEs, $p$--Laplacian,
nonlinear eigenvalue problems,   degenerated PDEs, quasilinear PDEs}

{\small{\bf Mathematics Subject Classification (2010)}:  Primary
26D10; Secondary 35D30, 35J60, 35R45 }

\vfill

{\noindent\small\bf Corresponding author:} Pavel Dr\'ab{e}k,
tel. no. +420-377 632 648
fax. no. +420-377 632 602

 \newpage

\section{Introduction}

In this paper we investigate the nonnegative solutions $u:\Omega\rightarrow \r$ to the partial differential inequality (PDI):
\begin{equation}\label{pdi-glowne}
-\Delta_{p,a}u\ge b(x)\Phi(u),
\end{equation}
where $\Omega\subseteq\r^n$ is an arbitrary open domain, $p>1$, the operator \(
\Delta_{p,a}u=\mathrm{div}(a(x)|\nabla u|^{p-2}\nabla u)\) is the degenerated $p$--Laplacian involving a weight function
$a(\cdot):\O\rightarrow [0,\infty)$, $b(\cdot)$ is a measurable function defined on $\Omega$, and $\Phi:[0,\infty)\rightarrow [0,\infty)$ is a given continuous function.


One of our main results, Theorem~\ref{theoHardy}, says that if the nonnegative function~$u$ solves~\eqref{pdi-glowne}, then we can apply it to construct the family of Hardy--type inequalities of the form:
\begin{equation*}
\int_\O \ |\xi|^p \mu_1(dx)\le \int_\O |\nabla \xi|^p\mu_2(dx),
\end{equation*}
where the measures $\mu_1$ and $\mu_2$ involve $u$ and the other quantities from~\eqref{pdi-glowne}, and $\xi$ is an arbitrary Lipschitz compactly supported  function defined on $\Omega$.
Those inequalities are constructed as a direct consequence of the Caccioppoli--type estimate for solutions to~\eqref{pdi-glowne} derived in Theorem~\ref{corocac}.

Our purpose is to investigate the two following issues: the qualitative theory of solutions to nonlinear problems and derivation of precise Hardy--type inequalities. We contribute to the first of them by obtaining Caccioppoli--type estimate for a priori not known solution, which in general is an important tool in the regularity theory. In the second issue we assume that the solution to~\eqref{pdi-glowne} is known and we use it in construction of Hardy--type inequalities. Substituting $a\equiv 1$  in our considerations, we retrieve several results obtained by the third author in~\cite{plap}, where she dealt with the partial differential inequality of the form $-\Delta_{p}u\ge \Phi$, admitting the function $\Phi$ depending on~$u$ and $x$. Some of the inequalities derived in~\cite{plap}, which motivated us to write this work, as well as those obtained here are precise
  as they hold with the best constants, see Remark~\ref{remconst}, Theorem~\ref{konstrukcje} and Theorem~\ref{theoHardysharp}.

The approach presented here and in the papers~\cite{plap}
and~\cite{orliczhardy} is the modification of methods from
\cite{nonex}. In all of these papers, the investigations start
with derivation of Caccioppoli--type estimates for the solutions
to nonlinear problem. The method  was inspired by the well known
nonexistence results by Pohozhaev and Mitidieri~\cite{pohmi_99}.

In contrast with the results from~\cite{{nonex},plap}, in this
paper we admit the degenerated $p$--Laplacian: $\Delta_{p,a}$
instead of the classical one in~\eqref{pdi-glowne}. Our main
results are the Caccioppoli--type estimate (Theorem~\ref{corocac})
and the Hardy--type inequality (Theorem~\ref{theoHardy}).
Some of the results obtained here are new even in the nondegenerated case $a\equiv 1$, see Remark~\ref{remNondeg} for details.

The discussion linking the eigenvalue problems with Hardy--type inequalities can be found in the paper by Gurka~\cite{gurka}, which generalized earlier  results by Beesack~\cite{beesack}, Kufner and Triebel~\cite{kufnertriebel}, Muckenhoupt~\cite{muckenhoupt}, and
 Tomaselli~\cite{tomaselli}. See also related more recent paper by   Ghoussoub and Moradifam~\cite{gm}.
Derivation of the Hardy inequalities on the basis of
supersolutions to $p$--harmonic differential problems  can be
found in papers by D'Ambrosio~\cite{dam1,dam2,dam3} and Barbatis,
Filippas, and Tertikas~\cite{bft1,bft2}. Other interesting results
linking the existence of solutions  in elliptic and parabolic
PDEs with Hardy type inequalities are presented
in~\cite{{anh},bargold,gaap,xiang,vz}, see also references
therein. We refer also to the recent contribution by the third
author  \cite{orliczhardy}, where,  instead of the nondegenerated
$p$--Laplacian in \eqref{pdi-glowne}, one deals with the
$A$--Laplacian: $\Delta_Au={\rm div}\left( \frac{A(|\nabla
u|)}{|\nabla u|^2}\nabla u\right)$, involving a function $A$ from
the Orlicz class. Similar estimates in the framework of nonlocal
operators can be found e.g. in~\cite{bogdan}.

Let us present several reasons to investigate the partial differential inequality of the form $-\Delta_{p,a}u\ge b(x)\Phi(u)$ rather than a simple one $-\Delta_{p}u\ge \Phi(u)$.

The first inspiration comes from the investigation of the Matukuma equation
\[
\Delta u + \frac{1}{1+|x|^2}u^q=0,\ q>1,
\]
which describes the dynamics of globular
clusters of stars~\cite{matukuma} and existence results for its generalized version, Matukuma--Dirichlet problems studied in~\cite{drabekgarciahuidobro} and reading as follows:
\begin{eqnarray*}
\left\{
\begin{array}{cc}
-{\rm div}\left( |x|^{\alpha }m(|\nabla u|)\nabla u \right) + \frac{|x|^{s-b}}{(1+|x|^b)^{s/b}}g(u)=0 & {\rm in} \ B(0,R),\\
u= 0 & {\rm on}\ \partial B(0,R).
\end{array}
\right.
\end{eqnarray*}

 Similar PDEs arise often in astrophysics to model several phenomena. For instance,  classical models of globular clusters of stars are modeled by  Eddington's equation~\cite{eddi}. Similar structure have models of dynamics of elliptic galaxies~\cite{bt}. Qualitative properties of solutions to the equations inspired by models and their generalizations, are considered e.g. in~\cite{bt,bat,ber, ciotti,drabekgarciahuidobro,puccimanahuidmana}.

The second motivation comes from functional analysis and it concerns the embeddings of $W^{1,p}_{a(\cdot)}(\Omega)$ into $L^s_{b(\cdot)}(\Omega)$ and its generalizations, when  Orlicz spaces are considered instead of $L^s_{b(\cdot)}(\Omega)$.
In such situation the equation
\begin{equation}\label{pattern}
-\mathrm{div}(a(x)|\nabla u|^{p-2}\nabla u)=\gamma b(x)|u|^{s-2}u
\end{equation}
is the Euler--Lagrange equation for the Rayleigh energy functional
\[
E(u)=\frac{\left(\int_\Omega |\nabla u(x)|^p
a(x)dx\right)^{\frac{1}{p}}}{\left(\int_\Omega | u(x)|^s
b(x)dx\right)^{\frac{1}{s}}}. \] The particular case of the
embedding $W^{1,p}_{a(\cdot)}(\Omega)\rightarrow
L^s_{b(\cdot)}(\Omega)$, where the weights are $a=|x|^\alpha,$
$b=|x|^\beta$, is the Caffarelli--Kohn--Nirenberg
inequality~\cite{caf-cohn-nir}.

The third reason to investigate solutions of degenerated PDEs is that even if we deal with equation like \eqref{pattern} in the case $a(x)\equiv 1$, and we know that its  solution $u(x)=w(|x|)$ is radial, we  can transform equation \eqref{pattern} into the related degenerated ODE involving two weights. For example, the equation
\begin{equation*}
-{\rm div}\left(t^{n-1}|v'(t)|^{p-2}v'(t)\right)= \gamma t^{n-1}|v(t)|^{p^*_{\beta}-2}v(t),
\end{equation*}
where $v(t)=w (r(t))$, $r(t)$ is inverse to $t(r)=\int_0^rs^{-\beta/p}ds =\frac{p}{p-\beta}r^{(p-\beta)/p}$, $p^*_\beta =p\frac{n-\beta}{n-p}$ is the Sobolev exponent in the embedding $W^{1,p}_{|x|^\beta}(\rn)\rightarrow L^{p^*_\beta}(\rn)$
  given by the Caffarelli--Kohn--Nirenberg inequality~\cite{caf-cohn-nir}, is related to the transformation of equation
\begin{equation}\label{cafaineq}
-\Delta_pu = \gamma |x|^{-\beta}|u|^{p^*_\beta-2}u,
\end{equation}
see e.g.~\cite{puccimanahuidmana} and the discussion on page 525 in~\cite{pucci}.

In many cases the solutions are known and therefore we can use them to construct Hardy--type inequalities. For example, it has been shown in~\cite[Theorem 5.1]{pucci}, that the function
\begin{equation}\label{talentiextre}
u(x)=c\left( 1+|x|^{\frac{p-\beta}{p-1}} \right)^{-\frac{(n-p)}{p-\beta}}\ {\rm where}\ c=\left[\frac{n-\beta}{\gamma}\left( \frac{n-p}{p-1} \right)^{p-1}  \right]^{\frac{(n-p)}{p-\beta}},
\end{equation}
is the solution of the  equation \eqref{cafaineq} in the case
of $\beta<p<n$. When $\beta=0$, we deal with Talenti extremal profile~\cite{talenti}.
This fact was the motivation for the analysis presented in the paper \cite{ak-is}, reported in Section \ref{sec:H}, where the authors, under ceratin assumptions, obtained  the inequality
\[
\bar{C}_{\gamma,n,p,r}\int_\rn \ |\xi|^p \left(1+r|x|^{\frac{p}{p-1}}\right)  \left(1+|x|^{\frac{p}{p-1}}\right)^{\gamma (p-1)-p}\, dx\le \int_\rn |\nabla \xi|^p\left(1+|x|^{\frac{p}{p-1}}\right)^{(p-1)\gamma}
\]
 in some cases with the best constants. Such inequalities in the case $p=2$ are of interest in the theory of nonlinear diffusions, where one
 investigates the  asymptotic behavior of solutions of equation $u_t=\Delta u^m$,
      see~\cite{blanchet_07} and the related works \cite{blanchet_09,sharp,gm}.

It might happen that the solutions to the partial differential inequality or equation~\eqref{pdi-glowne}  are known to exist by some existence theory, but their precise form
      is not known. In such a situation, under certain assumptions, we are still able to construct the Hardy inequality of the type
\[
\int_\O \ |\xi|^p b(x)\, dx\le \int_\O |\nabla \xi|^pa(x)\, dx,
\]
which perhaps could be applied to study further properties of solutions.
For example, the Hardy--Poincar\'{e} inequalities
like above, where $a(\cdot )=b(\cdot )$, are often equivalent to the solvability of degenerated PDEs of the type $${\rm div}\left(a(x)|\nabla u(x)|^{p-2}\nabla u(x)\right)=x^*,$$ where $x^*$ is an arbitrary functional on weighted Sobolev space $W^{1,p}_{\varrho,0}(\Omega)$ defined as the completion of $C_0^\infty(\Omega)$ in the norm of Sobolev space $W^{1,p}_{\varrho}(\Omega)$, see Theorem~7.12 in~\cite{akiraj}.

We hope that by the investigation of  the qualitative properties of  supersolutions to degenerated PDEs  and by
 constructions of Hardy--type inequalities, we can get deeper insight into the theory of degenerated elliptic PDEs.

\section{Preliminaries}\label{prelim}

\subsection*{Basic notation}

In the sequel we assume that $p>1$, $\O\subseteq\rn$ is an open subset not necessarily bounded.
By  $a(\cdot)-p$--harmonic problems we understand those which involve  degenerated  $p$--Laplace operator:
\begin{equation}\label{plaplacian}
\Delta_{p,a}u=\mathrm{div}(a(x)|\nabla u|^{p-2}\nabla u),
\end{equation}
with some nonnegative function $a(\cdot)$. The derivatives which appear in~\eqref{plaplacian} are understood in a distributional sense.
By $D {'}(\O)$ we denote the space of distributions defined on $\O$. If $f$ is defined on $\O$, then by $f\chi_{\O}$ we understand a function defined on $\rn$ which is equal to $f$ on $\Omega$ and
 which is extended by $0$ outside $\O$. Negative part of $f$ is denoted by $f^{-}:= {\rm min}\{ f,0\}$, while positive one by $f^{+}:= {\rm max}\{ f,0\}$. Moreover, every time when we deal with infimum, we set $\inf\emptyset=+\infty$.

\subsection*{Weighted Beppo Levi and Sobolev spaces}

\textbf{$B_p$ weights. } We deal with the special class of measures belonging to the  class  $B_p(\Omega)$.

\begin{defi}[Classes $W(\Omega)$ and $B_p(\Omega)$]\label{weightclass}
\rm
Let $\Omega\subseteq \mathbb{R}^{n}$ be an open set and let $\mathcal{M}(\Omega)$ be the set of all Borel  measurable real  functions defined on $\Omega$.
 Denote
$W(\Omega):= \left\{\varrho\in \mathcal{M}(\Omega):  \ 0<\varrho(x)<\infty,\ {\rm for\ a.e.}\  x\in\Omega \right\},$  and let $p>1$.
We will say that a weight $\varrho\in W(\Omega)$ satisfies $B_p(\Omega)$--condition
($\varrho\in B_p(\Omega)$ for short) if $
\varrho^{-1/(p-1)}\in L^1_{{  loc}}(\Omega).$
\end{defi}


\noindent The H\"older inequality leads to the following simple observation based on Theorem 1.5 in \cite{kuf-opic}. For readers' convenience we enclose the proof.

\begin{prop}\label{Lp:inclu:L1}
Let $\Omega\subset \mathbb{R}^n$ be an open set, $p>1$ and $\varrho\in B_p(\Omega)$. Then
\(
L^{p}_{\varrho,loc}(\Omega)\subseteq  L^1_{{\rm loc}}(\Omega)
\) and when $u_k\to u$ locally in $L^{p}_{\varrho}(\Omega)$ then also $u_k\to u$ in $L^1_{{ loc}}(\Omega)$.
\end{prop}

\noindent
{\bf Proof.} For any $\O{'}\subseteq\O$ such that $\overline{\O{'}}\subseteq\O$ and any $u\in L^p_{\varrho,loc}(\O)$
\[
\int_{\O{'}}|u|dx=\int_{\O{'}}|u|\varrho^{\frac{1}{p}}\varrho^{-\frac{1}{p}}dx\le \left(\int_{\O{'}}|u|^p\varrho dx \right)^{\frac{1}{p}} \left(\int_{\O{'}}\varrho^{-\frac{1}{p-1} } dx \right)^{1-\frac{1}{p}}<\infty.
\] The substitution of $u_k-u_l$ instead of $u$ implies second part of the statement.
\hfill$\Box$

\noindent\textbf{Weighted Beppo Levi space. } Assume that $\varrho(\cdot)\in B_p(\O)$. We  deal with the weighted Beppo Levi space
\[
{\cal L}^{1,p}_{\varrho}(\Omega):= \{ u\in D{'}(\Omega): \frac{\partial u}{\partial x_i} \in L^p_\varrho(\O)\ {\rm for}\ i=1,\dots,n\}.
\]
According to the above proposition and  \cite[Theorem 1, Section 1.1.2]{ma}, we have ${\cal L}^{1,p}_{\varrho}(\Omega)\subseteq W^{1,1}_{loc}(\O)$.
We will also consider local variants of Beppo Levi spaces: ${\cal L}^{1,p}_{\varrho,loc}(\Omega):= \{ u\in D{'}(\Omega): \int_{\Omega{'}}|\nabla u(x)|^p \varrho(x)dx <\infty \}$,
whenever $\overline{\Omega{'}}$ is a compact subset of $\Omega$. As it is also a subset in $W^{1,1}_{loc}(\O)$,
integration by parts formula applies to elements of ${\cal L}^{1,p}_{\varrho,loc}(\Omega)$ in the usual way.

\noindent\textbf{Two-weighted Sobolev spaces. } Let $\varrho_1(\cdot)\in W(\Omega),\varrho_2(\cdot)\in B_p(\O)$. We consider the space $W^{1,p}_{(\varrho_1,\varrho_2)}(\O)= L^p_{\varrho_1}(\O)\cap{\cal L}^{1,p}_{\varrho_2}(\Omega)$, i.e.
\begin{equation}\label{polnorma}
W^{1,p}_{(\varrho_1,\varrho_2)}(\Omega):= \left\{ f\in L^{p}_{\varrho_1}(\Omega)\cap D^{'}(\Omega) :  \frac{\partial f}{\partial x_1} , \dots ,\frac{\partial f}{\partial x_n} \in L_{\varrho_2}^p (\Omega )\right\},
\end{equation}
with the norm
$\| f\|_{W^{1,p}_{(\varrho_1,\varrho_2)}(\O)}:= \| f\|_{L^{p}_{\varrho_1}(\O)} + \| \nabla f\|_{L^{p}_{\varrho_2}(\O)}$.

\begin{prop}[\cite{kuf-opic}]\label{wlasn}
Let $p>1$, $\Omega\subseteq \mathbb{R}^{n}$ be an open set and $\varrho_1(\cdot)\in W(\Omega)$, $\varrho_2(\cdot)\in B_p(\O)$. Then  $W^{1,p}_{(\varrho_1,\varrho_2)}(\Omega)$ defined by \eqref{polnorma} equipped with the norm $\| \cdot\|_{W_{(\varrho_1,\varrho_2)}^{1,p}(\Omega)}$ is a Banach space.
\end{prop}

When $ \varrho_1 \equiv  \varrho_2 $, we deal with the usual weighted Sobolev space $W^{1,p}_{\varrho_1}(\O)$.
By   $W_{(\varrho_1,\varrho_2),0}^{1,p}(\Omega)$ we denote the completion of $C_0^\infty (\Omega)$ in the space $W_{(\varrho_1,\varrho_2)}^{1,p}(\Omega)$ and  we use the standard notation $W_{(\varrho_1,\varrho_1),0}^{1,p}(\Omega)=W^{1,p}_{\varrho_1,0}(\O)$ when $\varrho_1=\varrho_2$.

\noindent\textbf{ Some additional facts}

Having an arbitrary function $u\in W^{1,1}_{loc}(\Omega)$ (local Sobolev space), we define its value at every point $x\in\Omega$ by the formula 
\begin{equation}
\label{punkt}
u(x):= \limsup_{r\to 0}\barint_{B(x,r)} u(y)dy.
\end{equation}

 \begin{lem}[e.g. \cite{nonex}, Lemma 3.1]\label{crit}
Let $u\in W^{1,1}_{loc}(\O)$ be defined everywhere by~\eqref{punkt} and let $t\in \r$ be given. Then $
\{ x\in \rn : u(x)=t \}\subseteq \{ x\in\rn :\nabla u(x)=0\}\cup N,$
where $N$ is a set of Lebesgue's measure zero.
\end{lem}

\subsection*{Degenerated $p$--Laplacian}

Assume that $p>1$, $a\in B_p(\Omega)\cap L^1_{loc}(\Omega)$  (see Definition~\ref{weightclass}), and $u\in {\cal L}^{1,p}_{a,loc}(\O)$.
Then $a|\nabla u|^{p-1}\in L^1_{loc}(\Omega)$ as we have:
\begin{eqnarray*}
\int_{\Omega{'}}a|\nabla u|^{p-1}dx\le \left( \int_{\Omega{'} } adx\right)^{\frac{1}{p}} \left( \int_{\Omega{'}}|\nabla u|^p adx\right)^{1-\frac{1}{p}} <\infty,
\end{eqnarray*}
whenever $\Omega{'}$ is a compact subset of $\Omega$. In particular, $a|\nabla u|^{p-2}\nabla u \in L^1_{loc}(\Omega,\r^n)$ and so the weak divergence of $a|\nabla u|^{p-2}\nabla u \in L^1_{loc}(\Omega,\r^n)$
denoted by $\Delta_{p,a}u$ is well defined via the formula
\begin{equation}\label{placp}
\langle \Delta_{p,a}u,w \rangle = \langle {\rm div} \left(  a|\nabla u|^{p-2}\nabla u \right),w \rangle := -\int_\Omega a|\nabla u|^{p-2}\nabla u\cdot \nabla w dx
\end{equation}
where $w\in C_0^\infty (\Omega)$.
Obviously, in the case $a\equiv 1$ the operator $\Delta_{p,a}u$ reduces to the usual $p$--Laplacian ${\rm div} \left(  |\nabla u|^{p-2}\nabla u \right)$. It particular, it coincides with the Laplace operator in the case $p=2$.

\begin{rem}\label{remprzestrz}
\rm We observe that\begin{itemize}
\item[i)] as  $|\nabla u|^{p-2}\nabla u\in L^{\frac{p}{p-1}}_{a,loc}(\Omega,\rn)$, then the right--hand side in \eqref{placp} is well defined for every $w\in {\cal L}^{1,p}_a (\Omega)$ which is compactly supported in $\Omega$;

\item[ii)]  when $u\in {\cal L}^{1,p}_a(\Omega)$, formula~\eqref{placp} extends  for $w\in W^{1,p}_{(b,a),0}(\Omega)$, whenever $b\in W(\O)$. This follows from the estimates
\begin{eqnarray*}
|\langle \Delta_{p,a}u,w \rangle|&\le& \int_{\Omega}a|\nabla u|^{p-1}|\nabla w|dx= \int_{\Omega}(a^{\frac{1}{p {'}}}|\nabla u|^{p-1})(a^{\frac{1}{p}}|\nabla w|)dx\\
&\le&  \left( \int_{\Omega}|\nabla u|^p adx\right)^{1-\frac{1}{p}}  \left( \int_{\Omega}|\nabla w|^p adx\right)^{\frac{1}{p}}<\infty.
\end{eqnarray*}

Therefore, in that case $\Delta_{p,a}u$ can be also treated    as an element of $(W^{1,p}_{(b,a),0}(\Omega))^*$, the dual to the Banach space $W^{1,p}_{(b,a),0}(\Omega)$.
 We preserve the same notation  $\Delta_{p,a}u$ for this functional extension of formula~\eqref{placp}.
\end{itemize}

\end{rem}

\subsection*{Differential inequality}

Our analysis is based on the following differential inequality.

\begin{defi}\label{defnier}\rm
  Let   $a\in B_p(\Omega)\cap L^1_{loc}(\Omega)$ be a given weight, $u\in {\cal L}^{1,p}_{a,loc}(\O)$ be nonnegative, $\Phi: [0,\infty) \rightarrow [0,\infty)$ be a continuous function, $b(\cdot)$ be measurable and  $\Phi b\in L^1_{loc}(\Omega)$. Suppose further that  for every nonnegative compactly supported function $w\in
{\cal L}^{1,p}_a(\Omega)$ one has
\begin{equation*}
\int_\Omega \Phi(u) b(x)w\,dx >-\infty.
\end{equation*}
 We say that partial differential inequality (PDI for short)
\begin{equation}\label{nier}
-\Delta_{p,a} u\ge \Phi(u)b(x),
\end{equation}
holds if for every nonnegative compactly supported function $w\in
{\cal L}_a^{1,p}(\Omega)$ we have
\begin{equation}\label{nikfo}
 \langle-\Delta_{p,a} u,w \rangle
\ge  \int_\Omega \Phi(u)b(x) w\, dx,
\end{equation}
where $ \langle-\Delta_{p,a} u,w \rangle $ is given by~\eqref{placp}, see also Remark~\ref{remprzestrz}.
\end{defi}
We have the following observations.
\begin{rem}\label{mal1}$ $\rm\begin{itemize}
\item[i)] Inequality \eqref{nier} can be interpreted as a variant of $p$--superharmonicity condition for the degenerated $p$--Laplacian defined by \eqref{plaplacian}.
\item[ii)] In the case of equation in~\eqref{nier}: $-\Delta_{p,a} u= \Phi(u)b(x),$ we deal with the solution of the nonlinear eigenvalue problem.
\item[iii)] When $u\equiv D\ge 0$ is a constant on some subdomain $\O'\subseteq\O$,  inequality \eqref{nier} implies  $b(x)\Phi (D)\le 0$~a.e. in~$\Omega'$, equivalently this means that either $b\le 0$~a.e. in~$\Omega'$ and $\Phi(D)>0 $  or else $\Phi(D)=0$. Consequently, when $\Phi(0)=0$ and $\Phi(t)>0$ for $t>0$, inequality \eqref{nier}
holds on $\O_1$ with $u\equiv D$ if either $D\neq 0$ and $b\le 0$ on $\O_1$  or $D=0$ and $b$ is arbitrary.
\end{itemize}
\end{rem}

\subsection*{Assumption A}

By Assumption A we mean the set of conditions: ($a,b$), ($\Psi,g$), ($u$), and a)--d) below.
\begin{description}
\item[ ($a,b$)] $a\in L^1_{loc}(\Omega)\cap B_p(\Omega)$,  $b(\cdot)$ is measurable;
\item[ ($\Psi,g$)] The couple of continuous functions ($\Psi,g): (0,\infty) \times (0,\infty) \rightarrow (0,\infty) \times (0,\infty)$, where $\Psi$ is   Lipschitz on every closed interval in $(0,\infty)$,
 satisfy the following  compatibility conditions:
\begin{description}
\item[i)] the inequality \begin{equation}\label{Psinier111}g(t){\Psi}{'}(t)\le -C {\Psi} (t)\quad \mathrm{  a.e.\quad in\quad }(0,\infty)
\end{equation}
holds with some constant $C\in \r$ independent of $t$ and $\Psi$ is monotone (not necessarily strictly);
\item[ii)] each of the functions \begin{equation}\label{Theta}
t\mapsto\Theta(t):=\Psi(t){g^{p-1}(t)},\ \ {\rm and}\ \    t\mapsto \Psi(t)/g(t)
\end{equation} is nonincreasing or bounded in some neighbourhood of $0$.
\end{description}
\item[ ($u$)]
We assume that $u\in {\cal L}^{1,p}_{a,loc}(\O)$ is nonnegative, ($a,b$) holds, $\Phi: [0,\infty) \rightarrow [0,\infty)$ is a continuous function,
such that  for every nonnegative compactly supported function $w\in {\cal L}^{1,p}_a(\Omega)$ one has
$\int_\Omega \Phi(u) b(x)w\,dx >-\infty $ and $\Phi(u)b\in L^1_{loc}(\O)$.\\ Moreover, let us consider the set ${\cal A}$ of those $\s\in\r$ for which
\begin{equation}\label{znakphi}
\Phi(u)b(x)+\s\,\frac{ a(x)}{g(u)}|\nabla u|^p\ge 0 \quad \mathrm{  a.e.\ in\ }\Omega\cap \{ u>0\}.
\end{equation}
We suppose that
\begin{equation}\label{s0}
\s_0:= \inf {\cal A}=  \inf  \left\{\s\in\r :\ \s\ {\rm satisfies}\ \eqref{znakphi} \right\}\in \r.
\end{equation}
Since ${\rm inf}\, \emptyset =+\infty$,  ${\cal A}$ can be neither an empty set nor unbounded from below.

\item[a)] We suppose that  ($\Psi,g$) and ($u$) hold. Parameter $\sigma$ satisfies $\s_0\le \s<C$, where  $C$ is given by~\eqref{Psinier111} and $\s_0$ by~\eqref{s0}.
\item[b)] We suppose that  ($u$) and ($\Psi,g$) hold and  we assume that for every $R>0$ we have $b^{+}(x)(\Phi\Psi) (u)\chi_{0<u\le R}\in L^1_{loc}(\Omega)$.
\item[c)] We suppose that  ($u$)  and ($\Psi,g$) hold. When the set $\Omega_0:= \{ x: u(x)=0\}$ has a~positive measure, then we assume that at least one of the following conditions are satisfied \[\mathrm{ x)\ }  \Phi(0)=0,\qquad \mathrm{ y)\ }   b(x)\chi_{\Omega_0}\ge 0 ,\qquad \mathrm{ z)\ }  \lim_{\delta\to 0}\Psi (\delta)=0.\]
\item[d)] We suppose that   ($u$) and ($\Psi,g$) hold. We assume that for any compact subset $K\subseteq \O$ we have
\begin{eqnarray*}
&~& \Psi (R)\int_{K\cap \{ u\ge R/2\} }|\nabla u(x)|^{p-1}a(x)\, dx \stackrel{R\to\infty}{\rightarrow} 0,\\
 &~& \Psi (R)\int_{K\cap \{ u\ge R/2\} } \Phi (u)b(x)\, dx \stackrel{R\to\infty}{\rightarrow} 0.
\end{eqnarray*}

\end{description}

\subsection*{Comments on assumptions}
We have the following observations on Condition ($\Psi,g$).

\begin{rem}\label{Psi/gnonin}\rm \begin{itemize}
\item[i)] Assume that Condition ($\Psi,g$),~{i)} holds and, moreover, $g'(t)\ge-C$. Then $\left(  {\Psi}/{g} \right)'\le 0$ and $\Psi(t)/g(t)$ is nonincreasing.
\item[ii)] This condition is satisfied by pairs from Table~\ref{table}.
\item[iii)] For our purposes it suffices to weaken assumption ($\Psi,g$) in the following way. When $u(x)\in (k_1,k_2)\subseteq[0,\infty)$, we can restrict  Condition ($\Psi,g$) to $(k_1,k_2)$ instead of $(0,\infty)$. This follows from  the proofs presented below.
\end{itemize}\end{rem}
 \begin{center}
\begin{table}[h]
  \begin{center}
  \begin{tabular}{|c|c|c|c|}
    \hline
    $\Psi(t)$    &  $g(t)$  &   C   & remarks\\ \hline
    $t^{-\alpha}$ & $t$ & $\alpha$  & $\alpha\in\r$ \\\hline
    $\left(t\log (a+t)\right)^{-1}$ & $t\log (a+t)$ & $\log a$  & $a>1$ \\\hline
    $e^{-t}$ & bounded by $C$, $g'\ge -C$ & $C $ & $C>0$ \\\hline
    ${e^{-t}}/{t}$ & $t/(1+t)$ & $1$  & --- \\\hline
  \end{tabular}
  \end{center}
\caption{Example couples $(\Psi,g)$ which satisfy  Condition ($\Psi,g$). }\label{table}
\end{table}
\end{center}
The statement below shows that under Assumption A,$(u)$ the function $u$ cannot be constant almost everywhere in $\Omega$. Moreover, in many cases ${\cal A}$ is not empty and ${\rm inf}{\cal A}$ is a real number.
\begin{lem}
Suppose $u\in {\cal L}^{1,p}_{a,loc}(\O)$ is a nonnegative solution to the PDI  $-\Delta_{p,a} u\ge \Phi(u)b(x)$
in the sense of Definition~\ref{defnier}, under all assumptions therein. Moreover, let $b\ge 0$ a.e. in~$\Omega$. Then $\sigma_0$ given by~\eqref{s0} exists  and is finite if and only if $u$ is not a~constant function a.e. in~$\Omega$.
\end{lem}

\begin{proof} ($\Longleftarrow$) Assume that $u\not\equiv Const$. Then the set ${\cal A}$ is not empty as it contains zero, in particular $\s_0\le 0$. If $a(\cdot)>0,b(\cdot)\ge 0$ a.e.~in~$\Omega$, then the set ${\cal A}$  cannot be unbounded from below. Indeed, if ${\cal A}$ was unbounded from below, the inequality:
\[
\Phi(u(x))b(x)-\bar{n}\frac{ a(x)}{g(u(x))}|\nabla u(x)|^p\ge 0 \quad \mathrm{  a.e.\ in}\ \Omega\cap \{ u>0\}
\]
would hold for every $\bar{n}\in\n$. Consequently we could find $K_1,K_2>0$, such that
\[
\frac{1}{\bar{n}}\Phi(u(x))b(x)\ge \frac{ a(x)}{g(u(x))}|\nabla u(x)|^p\ge \frac{K_1}{K_2}>0
\]
a.e. in  $\{ u: |\nabla u|^p a(x)\ge K_1, g(u(x))\le K_2\}$, which is the set of positive measure and independent on $\bar{n}$. Taking the limit for $\bar{n}\to\infty$, we arrive at the contradiction.

($\Longrightarrow$) If $\sigma_0$ is a finite number, then $u$ cannot be constant. Indeed, for $u\equiv Const\ge 0 $, condition~\eqref{znakphi} implies ${\cal A}= (-\infty,\infty)$, which violates \eqref{s0}.
\end{proof}

\begin{rem}\rm
Assumption A,~d) is  satisfied in each of the following cases: \begin{itemize}
\item[i)] When $u$ is locally bounded.
\item[ii)] When $b\ge 0$, $u\in {\cal L}^p_{a,loc}(\O)$ and $\Psi (R)/R$ is bounded at infinity. Indeed, we have from  H\"older's inequality
\begin{eqnarray*}
Z_1(R)&:=&\Psi (R)\int_{K\cap \{ u\ge R/2\} }|\nabla u(x)|^{p-1}a(x)\, dx\\
& \le& \Psi(R)\left(  \int_{K\cap \{ u\ge R/2\} }|\nabla u(x)|^{p}a(x)\, dx \right)^{1-\frac{1}{p}} \left( \int_{K\cap \{ u\ge R/2\} }a(x)\, dx \right)^{\frac{1}{p}}\\
\end{eqnarray*}
and $Z_2(R):= \left(  \int_{K\cap \{ u\ge R/2\} }|\nabla u(x)|^{p}a(x)\, dx \right)^{1-\frac{1}{p}}\rightarrow 0$ as $R\to\infty$. On the other hand, by Czebyshev's inequality applied to $\mu(x)=a(x)dx$ on $K$, we get
\[
\int_{K\cap \{ u\ge R/2\} }a(x)\, dx =\mu(\{ x\in K :   u(x)\ge R/2\})\le \frac{2^p}{R^p}\int_K |u|^p a(x)dx =: \frac{1}{R^p} Z_3(R).
\]
Therefore, $Z_1(R)\le \frac{\Psi(R)}{R} Z_2(R)  Z_3(R)^{\frac{1}{p}}\to 0$ as $R\to\infty$.
\end{itemize}
\end{rem}

\section{Caccioppoli--type estimates}\label{CacSec}

Our first goal is to obtain the following  estimate. We call it  local, because it is stated on a part of the domain where $u$ is not bigger than a given $R$.

\begin{lem}[Local  estimate]\label{propcac} Suppose that   Assumption A holds except part d).  Assume further that $1<p<\infty$ and $u$ is a nonnegative solution to  PDI
\begin{equation}\label{dopisananierownosc}
-\Delta_{p,a} u\ge \Phi(u)b(x)
\end{equation}
 in the sense of Definition  \ref{defnier}.

Then  for any   nonnegative   Lipschitz function $\phi$ with
compact support in $\O$ such that the integral $\int_{\{\sf\cap \nabla u\neq 0\}}|\nabla\phi|^p\phi^{1-p}a(x)\,dx$ is finite and for any
 $R>0$ the inequality
\begin{eqnarray}
\label{localcaccc}
&~&\int_{\{ 0<u< R\}}\left(\Phi (u(x))b(x)
+\s\frac{a(x)}{g(u(x) )} |\nabla
u(x)|^p\right) \Psi(u(x))\phi(x)\, dx
\\
&\leq&c
\int_{\{\nabla u(x)\neq 0,\, 0<u<
R\} \cap {\rm supp}\phi}a(x)\Psi(u(x)) g^{p-1}(u(x)) |\nabla\phi (x)|^p\phi^{1-p}(x)\,dx
+\tilde{C}(R),\nonumber
\end{eqnarray}
holds, where $c:= \frac{1}{p^p}\left( \frac{p-1}{C-\sigma}\right)^{p-1}$,
\begin{equation*}
\tilde{C}(R)= \Psi(R)\left[\int_{\Oc\{
u\geq \frac{R}{2}\} }a(x)|\nabla u|^{p-1} | \nabla
{\phi}| \, dx -\int_{\Oc\{u\geq \frac{R}{2}\}}
\Phi(u)b (x) \phi\,dx\right] .
\end{equation*}
Moreover, all  quantities appearing in~\eqref{localcaccc}  are finite.
\end{lem}

The above result implies the following global estimate~\eqref{caccc} for  solutions to \eqref{dopisananierownosc}. It may be used to analyze various qualitative properties of them. We call it {\emph{Caccioppoli--type inequality}}, because  the right--hand side in~\eqref{caccc} does not
 involve  $\nabla u$ when we estimate $\chi_{\{ \nabla u\neq 0\}}$ by $1$, while, on the other hand, the left--hand side does involve $\nabla u$.

\begin{theo}[Caccioppoli--type estimate]
\label{corocac} Suppose that Assumption A  holds, $1<p<\infty$
 and $u
\in {\cal L}^{1,p}_{a, loc}(\O)$ is a nonnegative solution to the PDI  $$-\Delta_{p,a} u\ge \Phi(u)b(x)$$
in the sense of Definition  \ref{defnier}.

Then
for every nonnegative Lipschitz function $\phi$ with
compact support in $\O$ such that the integral $\int_{\sf}{|\nabla\phi|}^p\phi^{1-p}a(x)\, dx$ is finite,
we have  \begin{eqnarray}
&~& \int_{\O\cap \{ u>0\} } \left(\Phi(u(x))b(x) +\sigma{|\nabla u(x)|^p}\frac{a(x)}{g(u(x))}\right) \Psi(u(x))\phi (x)  dx \le~~~~~~~~~~~~~~~~\label{caccc}\\
 &&~~~~~~~~~~~~~~ c \int_{\O\cap \{ u(x)> 0,\nabla u(x)\neq 0\}\cap {\rm supp}\phi }
\,a(x)\Psi(u(x))g^{p-1}(u(x))|\nabla\phi (x)|^p\phi(x)^{1-p}dx,\nonumber
\end{eqnarray} with $c=\frac{(p-1)^{p-1}}{p^p(C-\s)^{p-1}}$.
\end{theo}

\begin{rem}\rm
Our assumptions do not exclude the case when measure $\left(\Phi(u(\cdot))b(\cdot) +\sigma{|\nabla u(\cdot)|^p}\frac{a(\cdot)}{g(u(\cdot))}\right) \Psi(u(\cdot))\chi_{ \{ u> 0\} }$ is equal to zero.
\end{rem}

\subsection*{Proof of the local estimates}

We use the following simple observations (see~\cite{plap}).

\begin{lem}  \label{lemmjednor}
Let $p>1,\ \t>0$ and $s_1,s_2\ge 0$, then
\[s_1s_2^{p-1}\le \frac{1}{p\t^{p-1}}\cdot  s_1^p +\frac{p-1}{p} \t\cdot s_2^p .\]
\end{lem}

\begin{lem}\label{niko-dym}
 Let $u,\phi$ be as in the assumptions of Lemma~\ref{propcac}. We fix
 $0<\delta<R$  and denote
\begin{eqnarray}\label{skalowanie}
u_{\delta,R}(x):= {\rm min}\left\{ u(x)+\delta, R \right\},
 & &  {G}(x):= \Psi(u_{\delta,R}(x)) \phi (x).
\end{eqnarray}
Then $u_{\delta,R}\in
{\cal L}^{1,p}_{a,loc}(\Omega )$, $ {G}\in 
{\cal L}^{1,p}_{a}(\Omega )$ and $ {G}$ is compactly supported in $\Omega$.\end{lem}

\begin{rem}\label{koment}$ $\rm
\begin{itemize}
\item[i)] We know that ${\cal L}^{1,p}_{a,loc}(\Omega)\subseteq W^{1,1}_{loc}(\Omega)$. This inclusion, together with Nikodym ACL Characterization Theorem~\cite[Section~1.1.3]{ma}, implies that we can verify if the function belongs to Sobolev space ${\cal L}^{1,p}_{a,loc}(\Omega)$
    by checking that it belongs to $W^{1,1}_{loc}(\Omega)$ and that
its derivatives computed almost everywhere belong to $L^p_{a,loc}(\Omega)$.
The fact that $\Psi$ is locally Lipschitz is used to apply Lemma~\ref{niko-dym} in order to ensure that $\Psi(u_{\delta,R}(x))$ belongs to $W^{1,1}_{loc}(\Omega)$.
\item[ii)]The nonnegativity of function $u$ allows to deduce that $ {G}\in
{\cal L}^{1,p}_{a}(\Omega )$. This fact  plays the crucial role in the proof of Lemma~\ref{propcac}.
\end{itemize}

\end{rem}

\subsubsection*{Proof of Lemma~\ref{propcac}}
We present the proof under the assumption that the set $\Omega_0$ in Assumption {  A},~c) has a positive measure. The proof in the case $u>0$ a.e.  follows by the simplification of the presented arguments.

\smallskip

 Let the quantities $\Phi,\Psi,g,a,b,u$ be as in \eqref{dopisananierownosc} and Assumption A, while $\phi$ be as in the statement of the lemma.

\noindent
The proof is performed in four  steps:

{\sc Step 1.} We prove that
for every $0<\delta <R$, the inequality
\begin{eqnarray}
 &~&\int_{\{ \O\cap u\le R-\delta \}}\left(\Phi(u)b(x) +\sigma\frac{a(x)}{g\left( u+\delta \right)}|\nabla u|^p\right){\Psi\left( u+\delta \right)}\phi~dx
  \nonumber\\~&\le&
c\int_{\O\cap \sf\cap \{ \nabla u\neq 0,\, 0<u\le R-\delta \}}a(x)
\Psi(u+\delta) g^{p-1}(u+\delta)
\left(\frac{|\nabla\phi|}{\phi}\right)^p\phi\,  dx \nonumber\\&+&
\tilde{C}(\delta,R) \ ~~~~~~~~~~ \label{cacccs1}
\end{eqnarray}
holds with $\sigma$ from Assumption A,~a) and
\[ \tilde{C}(\delta,R)=\Psi(R)\left[\int_{\Oc
\{ \nabla u\neq 0,\, u> R-\delta \} }a(x)|\nabla u|^{p-2} \nabla
u\cdot \nabla {\phi} \,  dx- \int_{\Oc\{u>R-\delta\}}
\Phi(u)b (x) \phi\,dx\right].
\]

{\sc Step 2.} We pass to the limit for $\delta\searrow 0$ and obtain
\begin{eqnarray*}
&~&\limsup_{\delta\searrow 0} c\int_{ \sf\cap \{ \nabla u\neq 0,\,  u\le R-\delta \}}a(x)
\Psi(u+\delta) g^{p-1}(u+\delta)
\left(\frac{|\nabla\phi|}{\phi}\right)^p\phi\,  dx +
\tilde{C}(\delta,R) \nonumber\\
&\le& c
\int_{\sf\cap \{\nabla u(x)\neq 0,\,0<u<
R\}}a(x)\Psi(u(x)) g^{p-1}(u(x))\, |\nabla\phi (x)|^p\phi^{1-p}(x)\,dx
+\tilde{C}(R).
\end{eqnarray*}

{\sc Step 3.}  For $\delta\geq 0$ we denote
\begin{eqnarray}\label{A_delta}
A_\delta (x) :=& \left(  \Phi (u)b(x) +\sigma \frac{a(x)}{g(u+\delta)} |\nabla u|^p\right)\Psi (u+\delta)\quad {\rm when}\ \delta >0,\\
A_0 (x) :=& \left(  \Phi (u)b(x) +\sigma \frac{a(x)}{g(u)} |\nabla u|^p\right)\Psi (u)\chi_{\{ u>0\} }\quad {\rm when}\ \delta =0.\nonumber
\end{eqnarray}

We show that
\begin{eqnarray*}
\liminf_{\delta\searrow 0}\int_{\{ 0< u\le R-\delta \}} A_\delta (x)\phi (x)dx \ge \int_{\{ 0< u< R \}} A_0 (x)\phi (x)dx.
 \end{eqnarray*}

{\sc Step 4.} We show that
\begin{eqnarray}
  \liminf_{\delta\searrow 0} \int_{\{ u=0 \}}A_\delta (x)\phi (x)dx \ge  0,\label{czwartykrok}
 \end{eqnarray}
which  implies the statement.

\subsubsection*{Proof of  {Step 1}}

Let us introduce the following notation for $J_i=J_i(\delta,R)$, $i=1,\dots,6$:
\begin{eqnarray*}
 J_1&=&  \int_{ \Oc\{ 0< u\le R-\delta\} }a(x) |\nabla u|^p {\Psi} {'}(u +\delta)\phi\,dx,\\
J_2&=&
\int_{ \Oc\{ 0<
u\le R-\delta \}}a(x)|\nabla u|^p\frac{\Psi\left( u+\delta \right)}{g\left( u+\delta \right)}\phi\,  dx,\\
J_3&=&
\int_{ \Oc\{ 0< u\le R-\delta \} }a(x)|\nabla
u|^{p-2} \, \Psi
(u+\delta)\, \nabla u\cdot \nabla {\phi} \, dx,\\
J_4&=&
\Psi(R)\int_{\Oc\{u>R-\delta\}}\Phi(u)b(x) \phi\,dx,\\
J_5&=&
\Psi(R)\int_{ \Oc\{ u> R-\delta
\} }a(x)|\nabla u|^{p-2}\nabla u\cdot \nabla {\phi} \,   dx,\\
J_6&=& \int_{ \sf\cap \{ \nabla u\neq 0,\, 0< u\le R-\delta \}}a(x)
\left(\frac{|\nabla\phi|}{\phi}\right)^p
\Psi(u+\delta) g^{p-1}(u+\delta)\phi\,  dx.
\end{eqnarray*}
By our assumptions all the above quantities are finite  (for $0 \leq u \leq R-\delta$ we have $\delta\leq u+\delta\leq R$).
Let $G$ be given by \eqref{skalowanie}. Choose $w:=G$ to be a test function in \eqref{nikfo}. Then the right hand side of \eqref{nikfo} becomes
\begin{eqnarray}
I&:=& \int_\O \Phi(u)b(x)   {G(x)}\,dx = \int_\O
\Phi(u)b(x) \Psi (u_{\delta,R})\phi\,dx=\nonumber\\
&=& \int_{\Oc\{u\leq R-\delta\}} \Phi(u)b(x)  \Psi(u+\delta)\phi
\,dx+
\Psi(R)\int_{\Oc\{u>R-\delta\}}\Phi(u)b(x) \phi\,dx=\nonumber\\
&=& \int_{\Oc\{u\leq
R-\delta\}}\Phi(u)b(x)  \Psi(u+\delta)\phi\,dx+J_4,\label{Idef}
\end{eqnarray}
so that $I$ is finite. Thus using \eqref{nikfo} we get the following estimate
\begin{eqnarray*}
I&=& \int_{\O} \Phi (u(x))b(x){G(x)}\,dx \le
\langle
-\mathrm{div}\left(a(x)|\na u|^{p-2}\na u\right), {G}\rangle \\
&=&\int_{\Oc\{\nabla u\neq 0\}}a(x) |\nabla u|^{p-2}\, \nabla u\cdot \nabla {G} \,  dx\\
& {=}&  \int_{ \Oc\{ \nabla u\neq 0,\, u\le R-\delta\} }a(x) |\nabla u|^p {\Psi}{'}(u +\delta)\phi\,dx+\\
&&+\int_{ \Oc\{ \nabla u\neq 0,\, u\le R-\delta \} }a(x)|\nabla
u|^{p-2}
\Psi(u+\delta)\, \nabla u\cdot \nabla {\phi} \,  dx +\\
&&+\Psi(R)\int_{\Oc \{ \nabla u\neq 0,\, u> R-\delta
\} }a(x)|\nabla u|^{p-2}\nabla u\cdot \nabla {\phi} \,   dx =J_1+J_3+J_5\\
&\le&
-CJ_2 +J_3+J_5.
\end{eqnarray*}
The last inequality follows from $J_1\le -CJ_2$ which holds due to
 \eqref{Psinier111}. Moreover,
\begin{eqnarray*}
 J_3&\le& \int_{ \Oc\{
\nabla u\neq 0,\, u\le R-\delta \}} a(x)| \nabla u|^{p-1} |
\nabla \phi |{\Psi (u+\delta)}\, dx=\\&=&
\int_{\sf\cap \{ \nabla u\neq 0,\, u\le R-\delta \}}
 \left(
\frac{|\nabla\phi|}{\phi} { g(u+\delta)} \right)  | \nabla u|^{p-1}a(x)
\frac{\Psi(u+\delta)}{g(u+\delta)} \,\phi dx.
\end{eqnarray*} We  apply  Lemma~\ref{lemmjednor} with $s_1=\frac{|\nabla\phi|}{\phi} g(u+\delta)$,
$s_2=|\nabla u| $ and arbitrary $\t>0$,  to get
\begin{eqnarray*}
J_3\le 
&& \frac{p-1}{p} \t\int_{ \sf\cap \{  \nabla u\neq
0,\, u\le R-\delta \}} a(x)|\nabla u|^p
\frac{\Psi(u+\delta)}{g(u+\delta)} \phi \, dx+
\\&+&\frac{1}{p\t^{p-1}}\int_{ \sf\cap \{ \nabla u\neq 0,\, u\le R-\delta \}}a(x)
\left(\frac{|\nabla\phi|}{\phi}\right)^p
\Psi(u+\delta) g^{p-1}(u+\delta)\phi\,  dx.\\
\le&&  \frac{p-1}{p} \t J_2
+\frac{1}{p\t^{p-1}} J_6.
\end{eqnarray*}
Combining these estimates we {deduce} that for $\tau >0$ such that $C - \frac{p-1}{p} \t=\s$ we have
\begin{eqnarray*}
I &\le&
-CJ_2+J_3+J_5\le\\
\ &\le&
\left(-C + \frac{p-1}{p} \t\right)J_2+\frac{1 }{p\t^{p-1}}J_6 + J_5 = -\sigma J_2+\frac{1 }{p\t^{p-1}}J_6 + J_5  .
\end{eqnarray*}
The last inequality and \eqref{Idef}
imply
 \begin{eqnarray*}
\int_{\Oc\{u\leq R-\delta\}}\Phi(u)b(x) \Psi(u+\delta)\phi\,dx+\sigma J_2
 \leq \frac{1}{p\t^{p-1}}J_6+J_5-J_4,
%
\end{eqnarray*}
 which implies~\eqref{cacccs1}, because $\tilde C(\delta,R)\ge J_5-J_4$ and $\tau = (C-\sigma)\frac{p}{p-1}$.

Introduction of parameters $\delta$ and $R$ is necessary as we
need  to move the quantities  $J_2,J_4$ in the estimates to the opposite sides
of inequalities. For doing this we have to know that
they are finite.

\subsubsection*{Proof of Step 2}
We show first that under our assumptions, when
 $\delta \searrow 0$, we have \begin{eqnarray}
&~&\int_{\sf\cap\{ \nabla u\neq 0,\,u+\delta\le R\} }a(x)\Psi(u+\d) g^{p-1}(u+\d) |\nabla\phi|^p\phi^{1-p}\,dx\label{deltadozera1}\\
&& ~~~~~~~~~~~~~~~~~~~~~~~~~~~~~~~~~~\longrightarrow\int_{\sf\cap \{ \nabla u\neq 0,\,0<u\le R\} }a(x)\Psi(u) g^{p-1}(u) |\nabla\phi|^p\phi^{1-p}\,dx. \nonumber
\end{eqnarray}
To verify this  we note that for a.e. $x\in\Omega$ we have
  \[
   \Psi(u(x)+\d) g^{p-1}(u(x)+\d)\chi_{\{\nabla u(x)\neq 0, u(x)
  +\delta\le R\}}\stackrel{\delta\to 0}{\to}  \Psi(u(x)) g^{p-1}(u(x))\chi_{\{ \nabla u(x)\neq 0,0<u(x)\le R\}}.
   \]
Indeed, when $0<u(x)<R$ or  $u(x)>R$    this follows from  the continuity  of the involved functions, while
   according to Lemma~\ref{crit}  the set $\{x: u(x)=0,\ |\nabla u(x)| \neq 0\}\cup \{x: u(x)=R,\ |\nabla u(x)| \neq 0\}$ is of measure zero.

For the proof of  \eqref{deltadozera1} we recall the nonnegative function $\Theta(t):=  \Psi(t) g^{p-1}(t)$ given by~\eqref{Theta}, which is  nonincreasing or bounded  in the neighbourhood of zero.

 Let us start with the first case, i.e. there exists $\ve>0$ such that for $t<\ve$ the function $\Theta(t)$ is nonincreasing. Without loss of generality we may assume $2\delta\le \ve\le R$ and
\[E_\ve= \left\{ \nabla u\neq 0, u< \frac{\ve}{2}\right\}\cap {\rm supp}\phi, \quad F_\ve= \left\{  \nabla u\neq 0,\frac{\ve}{2}\le u\right\}\cap {\rm supp}\phi. \]
Then we have
\[\int_{\sf\cap\{ \nabla u\neq 0,\,u+\delta\le R\} }\Theta(u+\delta)  a(x)|\nabla\phi|^p\phi^{1-p}\,dx=\]\[=\int_{E_\ve}\Theta(u+\delta)  a(x)|\nabla\phi|^p\phi^{1-p}\,dx+\int_{F_\ve}\Theta(u+\delta)\chi_{\{u+\delta\le R\}}  a(x)|\nabla\phi|^p\phi^{1-p}\,dx.\]

Let  us concentrate on the integral on $E_\ve$. We consider   $\delta<\ve/2$, so on $E_\ve$ we have $u+\delta<\ve$.
Note that mapping $t\mapsto \Theta(t)$ is nonincreasing for $t\in(0,\ve)$. For $\delta\searrow 0$ functions $\Theta_\delta (x):=\Theta(u(x)+\delta)$  converge to $\Theta(u(x))$ for almost every $x$. Therefore, due to  Lebesgue's Monotone Convergence Theorem we obtain
\[\lim_{\delta\searrow 0}\int_{E_\ve}\Theta(u+\delta)  a(x)|\nabla\phi|^p\phi^{1-p}\,dx=\int_{E_\ve}\Theta(u)  a(x)|\nabla\phi|^p\phi^{1-p}\,dx.\]

To deal with integrals over $F_\ve$ we note that
\[\Theta(u+\delta)\chi_{\{u+\delta\le R\}}  a(\cdot)|\nabla\phi|^p\phi^{1-p}\chi_{F_\ve}\le \chi_{\{ \ve/2 \leq u+\delta \leq R\}\cap {\rm supp}\phi}\Theta(u+\delta)  a(\cdot)|\nabla\phi|^p\phi^{1-p}\le\]\[\leq
 \sup_{t\in[\ve/2,R]}\Theta(t)\chi_{\sf}a(\cdot)|\nabla\phi|^p\phi^{1-p}\in L^1(\Omega).\]
Application of  Lebesgue's Dominated Convergence Theorem yields\[\lim_{\delta\searrow 0}\int_{F_\ve}\Theta(u+\delta)\chi_{\{u+\delta\le R\}}  a(x)|\nabla\phi|^p\phi^{1-p}\,dx=\int_{F_\ve\cap\{u< R\}}\Theta(u)  a(x)|\nabla\phi|^p\phi^{1-p}\,dx.\]
This completes the case of $\Theta$ decreasing in the neighbourhood of $0$.
The case of bounded~$\Theta$ follows from Lebesgue's Dominated Convergence Theorem (cf. as above for integral over
 $F_\ve$ with $\ve=0$).

To complete the proof of Step~2   we note that for $\delta\leq\frac{R}{2} $ we
have $\tilde C(\delta,R)\leq \tilde C(R).$

 \subsubsection*{Proof of Step 3}
We note that, when $A_\delta (x)$ is given by~\eqref{A_delta}, we have $A_\delta (x)\to A_0(x)$ a.e. in $\Omega_0$ as $\delta\searrow 0$, but we do not have information about the sign of $A_\delta$. Therefore we cannot apply for example Lebesgue's Dominated Convergence Theorem directly to justify the convergence of the integrals.
 Thus we distinguish between two cases: when $\sigma \ge 0$ and when $\s <0$. In both cases we prove the statement under each of the restrictions below
 on $\Psi$ and $\Psi/g$. They cover all the cases
 in Condition $(\Psi,g)$.\\
{\bf 3a)} $\Psi$ is nonincreasing and $\Psi/g$ is nonincreasing;\\
{\bf 3b)}  $\Psi$ is increasing and $\Psi/g$ is nonincreasing;\\
{\bf 3c)}  $\Psi$ is nonincreasing and $\Psi/g$ is bounded in some neighbourhood of $0$;\\
{\bf 3d)}   $\Psi$ is increasing and $\Psi/g$ is bounded in some neighbourhood of $0$.

\smallskip


 {\sc Case $\s\ge 0$.} In this case $\Psi$ is decreasing because $0\le \sigma <C$ by Assumption~A,~a). Therefore, we consider restrictions~3a) and~3c) only.

 Let us start with restriction 3a).
  Then $\Psi (u+\delta)\le \Psi (u)$, $\sigma\frac{\Psi (u+\delta)}{g(u+\delta)}\le \sigma\frac{\Psi (u)}{g(u)}$.
  Set
 \begin{eqnarray}
B_\delta(x):=  \left( b^{+}(x) \Phi(u) + \sigma\frac{a(x)}{g(u+\delta)}|\nabla u|^p\right)\Psi\left( u+\delta \right)\label{B_delta}.
\end{eqnarray}
Then $B_\delta \ge 0$ and we have
\begin{eqnarray*}
A_{\delta} (x) &=&   \left( b^{+}(x) \Phi(u) + \sigma\frac{a(x)}{g(u+\delta)}|\nabla u|^p\right){\Psi\left( u+\delta \right)}  +  b^{-}(x) \Phi(u)\Psi(u+\delta)\phi \nonumber\\
& =&  B_\delta (x) +  b^{-}(x) \Phi(u)\Psi(u+\delta)\ge  B_\delta (x) +  b^{-}(x) \Phi(u)\Psi(u).
\end{eqnarray*}
  Lebesgue's Monotone Convergence Theorem yields
\[
\lim_{\delta\searrow 0}\int_{\{ 0<u\le R-\delta \}} B_\delta (x)\phi (x)dx = \int_{\{ 0<u< R\}}  \left( b^{+}(x) \Phi(u) + \sigma\frac{a(x)}{g(u)}|\nabla u|^p\right){\Psi\left( u \right)}\phi (x)dx.
\]

For restriction 3c) we   verify the convergence of integrals involving $B_\delta$, given by~\eqref{B_delta}, by noticing that
\[
\lim_{\delta\searrow 0}\int_{\{ 0<u\le R-\delta \}} b^{+}(x) \Phi(u)\Psi(u+\delta)\phi (x)dx\rightarrow \int_{\{ 0<u< R\}} b^{+}(x) \Phi(u)\Psi(u)\phi (x)dx
\]
by Lebesgue's Monotone Convergence Theorem, while the convergence
\begin{eqnarray*}
\lim_{\delta\searrow 0}\int_{\{ 0<u\le R-\delta \}} {a(x)}|\nabla u|^p\frac{\Psi\left( u+\delta \right)}{g(u+\delta)} \phi (x)dx\rightarrow \\
\int_{\{ 0<u< R\}} {a(x)}|\nabla u|^p\frac{\Psi\left( u\right)}{g(u)} \phi (x)dx
\end{eqnarray*}
follows from Lebesgue's Dominated Convergence Theorem, as $\Psi/g$ is bounded near~$0$.

\smallskip
\noindent
{\sc Case $\s< 0$.} Let us consider first restriction 3a).  Then we have \[\sigma\frac{\Psi (u+\delta)}{g(u+\delta)}\ge \sigma\frac{\Psi (u(x))}{g(u(x))},\quad b^{-}(x)\Psi (u(x)+\delta)\ge b^{-}(x)\Psi (u(x))\] when $\delta>0$ and $u(x)>0$, and
\begin{eqnarray*}
A_\delta(x)&\ge& \Phi (u)b^{+}(x)\Psi (u+\delta) +\sigma \frac{a(x)}{g(u)}|\nabla u|^p\Psi (u) +\Phi (u)b^{-}(x)\Psi (u)\\
&=& \Phi (u)b(x)\Psi (u)+\sigma \frac{a(x)}{g(u)}|\nabla u|^p\Psi (u) + \Phi (u)b^{+}(x)\left(\Psi (u+\delta)- \Psi (u)\right)\\
& =& A_0(u) - \Phi (u)b^{+}(x)\left(\Psi (u)- \Psi (u+\delta)\right).
\end{eqnarray*}
Let us consider the integral over $\Omega$ from the  last expression and let $\delta$ converge to $0$. Note that $\Phi (u)b^{+}(x)\left(\Psi (u)- \Psi (u+\delta)\right)$ is nonnegative and decreasing to $0$ a.e. in $\Omega$ as $\delta \searrow 0$.  Moreover, according to  Assumption~{A},~b), we have
\begin{eqnarray*}
&~& \Phi (u)b^{+}(x)\left(\Psi (u)- \Psi (u+\delta)\right)\chi_{0< u\le R}\phi (x) \le b^{+}(x)\Phi (u)\Psi (u)\chi_{0< u\le R}\phi (x)\in L^1(\Omega).
\end{eqnarray*}
Lebesgue's Dominated Convergence Theorem gives
\[
\lim_{\delta\searrow 0} \int_{\{ 0<u\le R-\delta \}} \Phi (u)b^{+}(x)\left(\Psi (u)- \Psi (u+\delta)\right)\phi (x)dx  =0.
\]

If restriction 3b) applies we have $\sigma\frac{\Psi (u+\delta)}{g(u+\delta)}\ge \sigma\frac{\Psi (u)}{g(u)}$, $b^{+}(x)\Psi (u+\delta)\ge b^{+}(x)\Psi (u)$ when $u>0$, and then
\begin{eqnarray*}
A_\delta(x)&\ge& \Phi (u)b^{+}(x)\Psi (u) +\sigma \frac{a(x)}{g(u)}|\nabla u|^p\Psi (u) +b^{-}(x)\Phi (u)\Psi (u+\delta).
\end{eqnarray*}
Now the fact that
\[
\lim_{\delta\searrow 0} \int_{\{ 0<u\le R-\delta \}} (-b^{-}(x))\Phi (u)\Psi (u+\delta)\phi (x)\, dx\rightarrow \int_{\{ 0<u< R \}} (-b^{-}(x))\Phi (u)\Psi (u)\phi (x)\, dx
\]
follows from Lebesgue's Dominated Convergence Theorem because inequality $\Psi (u+\delta)\le \Psi (R)$ holds on this domain of integration and by Assumption {A},~(u).

In case of restriction 3c) we have  $b^{-}(x)\Psi (u+\delta)\ge b^{-}(x)\Psi (u)$, therefore
\begin{eqnarray*}
A_\delta(x)&\ge& \Phi (u)b^{+}(x)\Psi (u+\delta) +\sigma {a(x)}|\nabla u|^p \frac{\Psi (u+\delta)}{g(u+\delta)} +b^{-}(x)\Phi (u)\Psi (u).
\end{eqnarray*}
The convergence of integrals involving $\Phi (u)b^{+}(x)\Psi (u+\delta)$ follows from Lebesgue's Monotone Convergence Theorem, and the convergence of integrals involving ${a(x)}|\nabla u|^p \frac{\Psi (u+\delta)}{g(u+\delta)}$ follows from Lebesgue's Dominated Convergence Theorem,  because we can estimate
 $(\Psi/g)(u+\delta)\le {\rm sup}\{ (\Psi/g)(\lambda): \lambda\in (0,R)\}$ on domains of integration.

For restriction 3d) we use the following estimate for $u>0$:
\begin{eqnarray*}
A_\delta(x)&\ge& b^{+}(x)\Phi (u)\Psi (u) +\sigma {a(x)}|\nabla u|^p \frac{\Psi (u+\delta)}{g(u+\delta)} +b^{-}(x)\Phi (u)\Psi (u+\delta).
\end{eqnarray*}
We justify the convergence of  integrals from the expression on the right--hand side   by Lebesgue's Dominated Convergence Theorem using the fact that $\Psi (u+\delta)\le \Psi(R)$, $(\Psi/g)(u+\delta)\le {\rm sup}\{ (\Psi/g)(\lambda): \lambda\in (0,R)\}$
on the domain  of integration, and taking into account  Assumption~A, b).

 \subsubsection*{Proof of Step 4}
For almost every $x\in\Omega_0$ we have
\begin{eqnarray*}
&~&\left( \Phi(u (x))b(x) +\sigma \frac{a(x)}{g(u(x)+\delta)}|\nabla u(x)|^p\right)\Psi (u(x)+\delta)\phi(x) = \Phi(0)\Psi(\delta)b(x)\phi(x)\\
&=& \Psi(\delta)\left( b(x)\Phi(u)\chi_{\Omega_0}  \right)\cdot\phi(x)
\end{eqnarray*}
and $\left( b(x)\Phi(u)\chi_{\Omega_0}  \right)\cdot\phi(x)$ is integrable over $\Omega$ by Assumption~A, (u).
Since  Assumption~A, c)  holds we have either:  $\Phi(0)\Psi(\delta)b(x)\chi_{\Omega_0}\phi(x)\ge 0$ when x) or y) holds, or
 $\lim_{\delta\to 0} \Psi (\delta)=0$ in case z). In all cases \eqref{czwartykrok} holds.

This completes the proof of Lemma~\ref{lemmjednor}. \hfill$\Box$

\subsection*{Proof of Theorem~\ref{corocac} (Caccioppoli estimates)}
Assume at first that $\Psi$ is nonincreasing.  It suffices to let $R\to\infty$ in Lemma~\ref{propcac}. Without loss of generality we may assume that the
integral on the right--hand side of \eqref{caccc} is finite, as
otherwise the inequality follows trivially.  Since $a|\nabla
u|^{p-1}|\nabla {\phi}| $ and $\Phi(u)b\phi$ are integrable,  we have $\lim_{R\to \infty}
\widetilde{C}(R)=0$. Therefore,
\eqref{caccc} follows from \eqref{localcaccc} by Lebesgue's Monotone
Convergence Theorem.

When $\Psi$ is increasing we apply Assumption A,~d) and proceed similarly. \hfill$\Box$

\begin{rem}\rm\label{jeszcze-jeden} We can weaken the assumption of Lemma~\ref{propcac}, and thus in Theorem~\ref{corocac}, if we have more information about $u$. We suppose that $u\geq 0$ a.e. and $\Phi$ is continuous up to zero.
 In particular, we admit $u$  to be  equal to zero on a set of positive measure. If $u>0$ a.e. the assumption on $\Phi$ can be weakened. It suffices to consider continuous $\Phi :(0,\infty)\rightarrow (0,\infty)$ in Condition $(u)$ and omit Assumption~A,~c).  See Step~4
in the proof of Lemma~\ref{propcac}.
\end{rem}

\section{Hardy--type inequality}\label{sec:H}

As a direct consequence of Caccioppoli--type estimates for solutions to PDI, we obtain Hardy--type inequality for rather general class of test functions, i.e. Lipschitz and compactly supported functions. The following theorem implies several Hardy--type inequalities with the optimal constants, see Remark~\ref{remconst} below.

\begin{theo}[Hardy--type inequality]
\label{theoHardy} Suppose $a\in L^1_{loc}(\Omega)\cap B_p(\Omega)$,  $b\in L^1_{loc}(\Omega)$. Assume that $1<p<\infty$
 and $u
\in {\cal L}^{1,p}_{a, loc}(\O)$ is a nonnegative solution to the PDI  $-\Delta_{p,a} u\ge \Phi(u)b(x)$
in the sense of Definition  \ref{defnier}. Moreover, let Assumption A  hold.

Then
for every Lipschitz function $\xi\in {\cal L}^{1,p}_{a}(\Omega)$ with
compact support in $\O$
we have
\begin{equation}\label{hardyp}
\int_\O \ |\xi|^p \mu_1(dx)\le \int_\O |\nabla \xi|^p\mu_2(dx),
\end{equation}
where
\begin{eqnarray*}
&\mu_1(dx)&= \left(\Phi(u)b(x) +\sigma{|\nabla u|^p}\frac{a(x)}{g(u)}\chi_{\{ u\neq 0\}}\right) \Psi(u)\chi_{u>0}\, dx,
\\
&\mu_2(dx)&= \left(\frac{p-1}{C-\s}\right)^{p-1} a(x)\Psi(u)g^{p-1}(u)\chi_{\{ u> 0, \nabla u\neq 0\}}\, dx.
\end{eqnarray*}
\end{theo}

\begin{proof}
We apply of Theorem~\ref{corocac} with $\phi=\xi^{p}$, where $\xi$ is nonnegative Lipschitz function with compact support. Then {$\phi$ is Lipschitz and}
\[|\nabla \xi|^p=\left(\frac{1}{p}\phi^{\frac{1}{p}-1}|\nabla \phi|\right)^p=\frac{1}{p^p}\left(\frac{|\nabla \phi|}{\phi}\right)^p\phi.\]
{Therefore} \eqref{caccc} becomes \eqref{hardyp}. Note that for every  Lipschitz function $\xi$ with compact support in $\O$  we  have $\int_\O |\nabla \xi|^pa(x)\,dx<\infty$, equivalently $\int_{\sf}{|\nabla\phi|}^p\phi^{1-p}a(x)\, dx<\infty$.  As the absolute value of a Lipschitz function is a Lipschitz function as well, we write $|\xi|$ instead of $\xi$ on the left--hand side and do not require its nonnegativeness.
\end{proof}
 \begin{center}
\begin{table}[h]
  \begin{center}\small{
  \begin{tabular}{|c|c|l|}
    \hline
    Inequality & Optimality  &   Comment  \\ \hline
     classical Hardy & \cite{ko}  & proven in~\cite{plap} \\\hline
    Hardy--Poincar\'{e}  & \cite{bcp-plap} & via~\cite{plap}; improved constants from~\cite{blanchet_07,gm}; \\
       &  \cite{ak-is} & Theorem~\ref{konstrukcje} and Remark~\ref{remconst} here \\\hline
    Poincar\'{e}  &  Remark~7.6 in~\cite{akiraj} & concluded from~\cite{plap}\\\hline
    exponential--weighted  & expected & \cite[Theorem~5.5]{plap} vs.~\cite{akkppstudia}    \\
      Hardy                & expected & \cite[Theorem~5.8]{plap} vs.~\cite[Proposition~5.2]{akkppcentue}    \\\hline
  \end{tabular}}
  \end{center}
\end{table}
\end{center}
\begin{rem}\label{remconst}\rm
Let us point out that some of the inequalities derived previously in~\cite{plap}, which motivated us to write this work, are sharp as they hold with the best constants. Namely, they are achieved in
 the classical Hardy inequality (Section~5.1 in~\cite{plap});
 the Hardy--Poincar\'{e} inequality obtained  in~\cite{bcp-plap} due to~\cite{plap}, confirming some constants from~\cite{gm} and~\cite{blanchet_07} and establishing the optimal constants in further cases;
 the Poincar\'{e} inequality concluded from~\cite{plap}, confirmed to hold with best constant in Remark~7.6 in \cite{akiraj}.
Moreover, the  inequality in Theorem~5.5 in~\cite{plap} can also be retrieved by the methods from~\cite{akkppstudia} with the same constant, while some inequalities from Proposition~5.2  in~\cite{akkppcentue} are comparable with Theorem~5.8 in~\cite{plap}. In  Theorem~\ref{konstrukcje}, we provide some extensions of Hardy--Poincar\'{e} inequalities from~\cite{bcp-plap}, which are proven in \cite{ak-is} by applying the results obtained in this paper. Some of them hold with the optimal constants.
\end{rem}

\begin{rem}\rm
It is known \cite{akkambullpan} that Hardy inequalities can imply Gagliardo--Nirenberg interpolation inequalities for intermediate derivatives:
\[
\|\nabla u\|_{L^q(\Omega,\mu)}^2 \le C \| u\|_{L^r(\Omega,\mu)} \|\nabla^{(2)} u\|_{L^p(\Omega,\mu)} , {\rm where}\ \frac{2}{q}=\frac{1}{r} +\frac{1}{p}
\]
if one has Hardy inequality:
\(
\| u\|_{L^p(\Omega, \varrho \cdot \mu)}\le C\|\nabla u\|_{L^p(\Omega, \mu)}
\)
under certain assumptions on the measure $\mu$ and the weight function $\varrho$.
\end{rem}

\begin{rem}\rm
When we know  that $u$ is strictly positive almost everywhere, due to Remark~\ref{jeszcze-jeden}, the statement of Theorem~\ref{theoHardy} holds under the assumption that $\Phi :
(0,\infty)\rightarrow (0,\infty)$ is continuous in Condition $(u)$
and { we can} omit Assumption~A,~c).
\end{rem}

\begin{rem}\rm \label{remNondeg}
In the nondegenerated  case, i.e. when $a(\cdot)= b(\cdot)\equiv 1$,  Theorem~\ref{theoHardy}, as well as Theorem~\ref{corocac}, retrieves  the results of~\cite{plap}.
 In constrast with \cite{plap} our function $\Psi$ need not be increasing here. Hence, broader class of measures $\mu_1$ and $\mu_2$ may appear in~\eqref{hardyp}. Therefore our result generalizes that of~\cite{plap} even in nondegenerated case.
\end{rem}

\subsection*{Hardy inequalities resulted from existence theorems}
We are going to derive sharp Hardy type inequality, not knowing $u$ explicitly but only its existence.
We assume now that $b$ is nonnegative and  that there exists a nonnegative nontrivial locally bounded solution of PDI $-\Delta_{p,a} u\ge  b(x)u^{p-1}$ i.e.,
\begin{equation}\label{nikfosimple}
\begin{array}{ccc}
\langle -\Delta_{p,a} u, w\rangle
&\ge& \int_\Omega b(x)u^{p-1} w\, dx,
\end{array}
\end{equation}
holds for every nonnegative compactly supported function $w\in
{\cal L}_a^{1,p}(\Omega)$.

This is the special case of inequality \eqref{nier} for $\Phi (u)= u^{p-1}$. Our result reads as follows.

\begin{theo}[Sharp Hardy--Poincar\'{e} inequality]
\label{theoHardysharp}  Assume that $1<p<\infty$,  $a, b\in W(\Omega),$  $a\in L^1_{loc}(\Omega)\cap B_p(\Omega)$,
 and $u\in {\cal L}^{1,p}_{a, loc}(\O), bu^{p-1}\in L^1_{loc}(\O)$, $u$ is  a~nonnegative nontrivial solution to~\eqref{nikfosimple}.
Then for every Lipschitz function $\xi\in {\cal L}^{1,p}_{a}(\Omega)$ with
compact support in $\O$
we have
\begin{equation}\label{hardypaa11}
\int_\O \ |\xi|^p b(x)\, dx\le \int_\O |\nabla \xi|^pa(x)\, dx.
\end{equation}
Moreover, if there exists nontrivial, nonnegative,    $u_0\in W^{1,p}_{(b,a),0}(\Omega)$ which is the solution to $-\Delta_{p,a} u_0= b(x)u_0^{p-1}\in L^1_{loc}(\O)$ then inequality~\eqref{hardypaa11} is sharp, i.e. $C=1$ is the optimal constant in the inequality $C\int_\O \ |\xi|^p b(x)\, dx\le \int_\O |\nabla \xi|^pa(x)\, dx$.
\end{theo}
\begin{proof} We apply Theorem \ref{theoHardy} with $\Psi(t)=\frac{1}{t^{p-1}}$, $g(t)=t$, $\Phi(t)=t^{p-1}$, $\sigma =0$ and verify that under our conditions Assumption~A is satisfied. This gives~\eqref{hardypaa11}. Suppose now that there exists $u_0$ satisfying all the requirements of the theorem. Let us consider the sequence $(w_k)_{k\in \n}$ of smooth compactly supported functions, such that  $w_k\to u_0$ in   $W^{1,p}_{(b,a)}(\Omega)$.   Since each  $w_k$  has a compact support and belongs to ${\cal L}^{1,p}_{a}(\Omega)$, we have the equality $$\langle -\Delta_{p,a} u_0, w_k\rangle =\int_\Omega |\nabla u_0|^{p-2}\nabla u_0 \cdot\nabla w_k\, a(x)dx
= \int_\Omega b(x)u_0^{p-1} w_k\, dx.$$ When we let $k\to\infty$, we get  $\int_\Omega |\nabla u_0|^{p}\, a(x)dx
= \int_\Omega b(x)u_0^{p}\, dx$ which proves sharpness.
\end{proof}

\begin{rem}\rm
 Theorem~\ref{theoHardysharp}  is known in the case $a\equiv 1, b\equiv 1$, see~\cite{anane} or Remark~1 on page~163 in~\cite{lindqvist90}.
\end{rem}

\begin{rem}\rm
We substitute the special value of $\sigma =0$, in the proof of the above statement. Therefore, we do not expect that the inequality~\eqref{hardypaa11} holds with the best constant in general.
\end{rem}

\subsection*{Sharp Hardy--Poincar\'{e} inequalities with best constants}

Using the Talenti extremal profile given by~\eqref{talentiextre} where $\beta =0$ in our approach,
 one obtains the following theorem, cf.~\cite{ak-is} for details.

\begin{theo}\label{konstrukcje} Assume that $1<p<\infty$, $\gamma>1-\frac{n}{p}$, $0<r<1-\frac{p}{n}+ \gamma\frac{p}{n}$ and
$v_1(x):=\left(1+r|x|^{\frac{p}{p-1}}\right)  \left(1+|x|^{\frac{p}{p-1}}\right)^{\gamma (p-1)-p}$,
$v_2(x):= \left(1+|x|^{\frac{p}{p-1}}\right)^{(p-1)\gamma}$.
Then for every $\xi\in W^{1,p}_{v_1,v_2}(\rn)$ we have
\[
\bar{C}_{\gamma,n,p,r}\int_\rn \ |\xi|^p \left(1+r|x|^{\frac{p}{p-1}}\right)  \left(1+|x|^{\frac{p}{p-1}}\right)^{\gamma (p-1)-p}\, dx\le \int_\rn |\nabla \xi|^p\left(1+|x|^{\frac{p}{p-1}}\right)^{(p-1)\gamma},
\]
where $\bar{C}_{\gamma,n,p,r}= n\left( \frac{p}{p-1}\right)^{p-1}\left( \gamma-1 +\frac{n}{p}(1-r)\right)^{p-1}$. Moreover, constant $\bar{C}_{\gamma,n,p,r}$ is optimal when $\gamma >nr +1 -\frac{n}{p}$ and when $\gamma = 1+n(1-\frac{1}{p})$, $r=1$.
\end{theo}

\begin{rem} \rm Such inequalities in the case $p=2$ are very much of interest in the theory of nonlinear diffusions, where one
 investigates the  asymptotic behavior of solutions of the equation $u_t=\Delta u^m$,
      see~\cite{blanchet_07}.
To our best knowledge our inequalities are new if $r\neq 1$ in general. However,  as an example dealing with $r\neq 1$ and $p=2$  we refer to the fourth line on page~434 in~\cite{blanchet_07}, which is our case with $r=\gamma/n$.
Proof of that inequality in \cite{blanchet_07} requires knowledge about the best constants in Sobolev inequality, which we do not need. We can also prove Proposition~3 from~\cite{blanchet_07} by our methods and generalize it for an arbitrary~$p$.
\end{rem}

\begin{rem}\rm
There is a particular interest in the Hardy--Poincar\'{e} inequalities with decreasing weights (involving negative power  $\g<0$), which are not covered in~\cite{bcp-plap}. In our Theorem \ref{konstrukcje} we do allow some of such inequalities with optimal constants.
\end{rem}

The above statement can be compared with the following one obtained in \cite{bcp-plap}, which follows as the special case
of Theorem \ref{theoHardysharp} when one substitutes $r=1$. Consequently one has to assume that $\gamma >1$.

\begin{theo}[cf. \cite{bcp-plap}]\label{corolapP}
Suppose $p>1$ and $\gamma>1$. Then, for every  function $\xi\in W^{1,p}_{v_1,v_2}(\rn)$, where $ v_1(x)=\left(1+|x|^{\frac{p}{p-1}}\right)^{(p-1)(\gamma-1)},\ v_2(x)=\left(1+|x|^{\frac{p}{p-1}}\right)^{(p-1)\gamma} $, we have
\[
\bar C_{\gamma,n,p}\int_\rn \ |\xi|^p \left[(1+|x|^{\frac{p}{p-1}})^{p-1}\right]^{\gamma-1}dx \le \int_\rn |\nabla \xi|^p\left[(1+|x|^{\frac{p}{p-1}})^{p-1}\right]^{\gamma}\ dx,
\]
with
 $\bar C_{\gamma,n,p}={n}\left( \frac{p\left(\gamma-1\right)}{p-1}\right)^{{p-1}}$. Moreover, for $\gamma>n+1-\frac{n}{p}$, the constant $\bar C_{\gamma,n,p}$ is optimal.
\end{theo}

\section*{Acknowledgement}
A.K. and I.S. were supported by NCN grant 2011/03/N/ST1/00111. This work originated when I.S. visited University of West Bohemia in Pilsen in February 2013. She want to thank Pilsen University for hospitality. P.D. was supported by the Grant Agency of Czech Republic, Project No. 13--00863S.


\begin{thebibliography}{99}


\bibitem{anane} A. Anane, {\em Simplicit\'e et isolation de la premie\`re valeur propre du p-laplacien avec poids.  [Simplicity and isolation of the first eigenvalue of the p-Laplacian with weight],} C. R. Acad. Sci. Paris S\'er. I Math. {  305} (1987), no. 16, 725-728.

\bibitem{anh} C. T. Anh, T. D. Ke, {\em On quasilinear parabolic equations involving weighted p--Laplacian operators}, Nonlinear Differential Equations and Applications NoDEA 17 (2), 195--212.


\bibitem{bt} M. Badiale, G. Tarantello, {\em  A Sobolev--Hardy inequality with applications to a nonlinear elliptic equation arising in astrophysics}, Arch. Ration. Mech. Anal. 163 (2002) 259--293.

\bibitem{bargold} P. Baras,  A. J. Goldstein, {\em The heat equation with a singular potential}, Trans. Amer. Math. Soc. 284 (1) (1984), 121--139.

\bibitem{bft1} G. Barbatis, S. Filippas, A. Tertikas, {\em A unified approach to improved Lp Hardy inequalities with best constants,} Trans. Amer. Math. Soc. 356 (6) (2004), 2169--2196.

\bibitem{bft2} G. Barbatis, S. Filippas, A. Tertikas, {\em Series expansion for Lp Hardy inequalities,}  Indiana Univ. Math. J. 52 (1) (2003) 171--190.

\bibitem{bat}J. Batt, W. Faltenbacher, E. Horst, {\em Stationary spherically symmetric models in stellar dynamics}, Arch. Rational Mech. Anal. 93 (1986), 159--183.

\bibitem{beesack} P. R. Beesack, {\em Hardy's inequality and its extensions}, Pacific J. Math.  XI  1961.

\bibitem{ber} G. Bertin, {\em Dynamics of galaxies}, Cambridge University Press, Cambridge--New York (2000).



\bibitem{blanchet_07} A. Blanchet, M. Bonforte, J. Dolbeault, G. Grillo, J.--L. V\'{a}zquez, {\em Hardy--Poincar\'{e} inequalities and application to nonlinear diffusions}, C. R. Acad. Sci. Paris  344 (7) (2007), 431--436.

\bibitem{blanchet_09} A. Blanchet, M. Bonforte, J. Dolbeault, G. Grillo, J.--L. V\'{a}zquez, {\em Asymptotics of the fast diffusion equation via entropy estimates}, Archive for Rational Mechanics and Analysis 191 (2) 2009, revised 2007, 347--385.

\bibitem{bogdan} K. Bogdan, B. Dyda, P. Kim, {\em Hardy Inequalities and Non--explosion Results for Semigroups}, Potential Anal., DOI: 10.1007/s11118-015-9507-0.

\bibitem{sharp} M. Bonforte, J. Dolbeault, G. Grillo,  J.--L. V\'{a}zquez, {\em Sharp rates of decay of solutions to the nonlinear fast diffusion equation via functional inequalities},  PNAS  09/2010; 107(38):16459-64. DOI:10.1073/pnas.1003972107.

\bibitem{caf-cohn-nir}  L. Caffarelli, R. Kohn, L. Nirenberg, {\em First order inequalities with weights}, Compos. Math. 53 (1984) 259--275.

\bibitem{ciotti} L. Ciotti, {\em Dynamical models in astrophysics}, Lecture Notes, Scuola Normale Superiore, Pisa (2001).


\bibitem{dam1} L. D'Ambrosio, {\em Hardy type inequalities related to degenerate elliptic differential operators,} Ann. Sc. Norm. Super. Pisa Cl. Sci. (5) IV (2005), 451--486.

\bibitem{dam2} L. D'Ambrosio, {\em Some Hardy inequalities on the Heisenberg group,} Differ. Uravn. 40 (4) (2004), 509--521. (in Russian); translation in Differ. Equ. 40 (4) (2004), 552--564.

\bibitem{dam3} L. D'Ambrosio, {\em Hardy inequalities related to Grushin type operators,} Proc. Amer. Math. Soc. 132 (3) (2004),  725--734.


\bibitem{akiraj}  R. N. Dhara, A. Ka\l{}amajska, {\em On equivalent conditions for the validity of Poincar\'e inequality on weighted Sobolev space with applications to the solvability of degenerated PDEs involving p--Laplacian,} J. Math. Anal. Appl. 432 (1) (2015), 463--483.


\bibitem{drabekgarciahuidobro} P. Dr\'{a}bek, M. Garc{\'\i}a--Huidobro, R. Man\'{a}sevich, {\em Positive solutions for a class of equations with a $p$--Laplace like operator and weights},  Nonlinear Anal. TMA 71 (3--4) (2009), 1281--1300.

\bibitem{eddi} A. S. Eddington, {\em The dynamics of a globular stellar system}, Monthly Notices Royal Astronom. Soc., 75 (1915), 366--376.


\bibitem{gaap} J. P. Garc\'{i}a--Azorero, I. Peral--Alonso, {\em Hardy inequalities and some critical elliptic and parabolic problems,} J. Diff. Eq., { 144} (1998), 441--476.

\bibitem{gm} N. Ghoussoub, A. Moradifam, {\em Bessel pairs and optimal Hardy and Hardy--Rellich inequalities,} Math. Ann. 349 (1) (2011), 1--57.

\bibitem{gurka} P. Gurka, {\em Generalized Hardy's inequality,} \v{C}asopis P\v{e}st. Mat. 109 (2) (1984),  194--203.








\bibitem{nonex} A. Ka\l amajska, K. Pietruska-Pa\l uba, I. Skrzypczak, {\em Nonexistence results for differential inequalities involving $A$--Laplacian,} Adv. Diff. Eqs. 17 (3--4) (2012), 307--336.

\bibitem{akkppstudia} A. Ka\l amajska, K. Pietruska-Pa\l uba, {\em On a variant of Hardy inequality between weighted Orlicz spaces,} Studia Math. 193 (1) (2009),  1--28.

\bibitem{akkppcentue} A. Ka\l amajska, K. Pietruska-Pa\l uba, {\em New Orlicz variants of Hardy type inequalities with power, power--logarithmic, and power--exponential weights,} Cent. Eur. J. Math. 10 (6) (2012),  2033--2050.

\bibitem{akkambullpan}  A. Ka\l amajska, K. Pietruska-Pa\l uba, {\em On a variant of Gagliardo--Nirenberg inequality deduced from Hardy}, Bull. Pol. Acad. Sci. Math. 59 (2) (2011), 133--149.

\bibitem{ak-is} A. Ka\l amajska, I. Skrzypczak, {\em Constructions of  Hardy--Poincar\'{e} inequalities  by using the Talenti extremal profiles}, preprint 2016.


\bibitem{ko} A. Kufner, B. Opic, {\em Hardy--type Inequalities,}
Longman Scientific and Technical, Harlow, 1990.

\bibitem{kuf-opic} A. Kufner, B. Opic, {\em How to define reasonably weighted Sobolev spaces,} Comment. Math. Univ. Carolin. {25} (3) (1984), 537--554.

\bibitem{kufnertriebel} A. Kufner, H. Triebel, {\em Generalization of Hardy's inequality}, Conf. Sem. Mat. Univ. Bari 156 (1978), 21 pp. (1979).



\bibitem{lindqvist90} P. Lindqvist, {\em On the equation ${\rm div} \left(|\nabla u|^{p-2}\nabla u\right)+ \lambda |u|^{p-2}u= 0$}, Proceedings of the American Mathematical Society, {  109} (1)  (1990), 157--164.

\bibitem{matukuma} T. Matukuma,{\em The Cosmos,} Iwanami Shoten, 1938.

\bibitem{ma} V. G.  Maz'ya, {\em Sobolev Spaces}, Springer--Verlag, Berlin, 1985.

\bibitem{muckenhoupt} B. Muckenhoupt, {\em Hardy's inequality with weights}, Studia Math. 44 (1972), 31--38.

\bibitem{pohmi_99} E. Mitidieri, S. Pohozaev,  {\em  Nonexistence of positive solutions to quasilinear elliptic problems in $\rn$,} Proc. Steklov. Inst. Math.,  227  (1999), 186--216, (translated from Tr. Mat. Inst. Steklova { 227} (1999), 192--222.

\bibitem{pucci}  P. Pucci, R. Servadei, {\em  Existence, non--existence and regularity of radial ground states for p--Laplacian equations with singular weights,} Ann. Inst. H. Poincar\'{e} Anal. Non Lin\'{e}aire 25 (3) (2008), 505--537.

\bibitem{puccimanahuidmana}  P. Pucci, M. Garc\'{i}a--Huidobro, R. Manaevich, J. Serrin, {\em Qualitative properties of ground states for singular elliptic equations with weights,} Ann. Mat. Pura Appl.  185 (4) (2006) 5205--5243.




\bibitem{plap} I. Skrzypczak, {\em Hardy--type inequalities derived from $p$--harmonic problems}, Nonlinear Anal. TMA  93  (2013), 30--50.

\bibitem{bcp-plap} I. Skrzypczak, {\em Hardy--Poincar\'{e}--type inequalities derived from $p$--harmonic problems,} Banach Center Publ. 101 `Calculus of variations and PDEs' (2014), 223--236.

\bibitem{orliczhardy} I. Skrzypczak, {\em Hardy inequalities resulted from nonlinear problems dealing with $A$--Laplacian,} NoDEA Nonlinear Differential Equations Appl. 21 (6) (2014), 841--868.

\bibitem{talenti} G. Talenti, {\em Best constant in Sobolev inequality,} Ann. Mat. Pura Appl.  110 (4) 1976, 353--372.

\bibitem{tomaselli} G. Tomaselli, {\em A class of inequalities}, Boll. Un. Mat. Ital. (4) 2 1969 622--631.

\bibitem{xiang} C.--L. Xiang, {\em Asymptotic behaviors of solutions to quasilinear elliptic equations with critical Sobolev growth and Hardy potential,} J. Differential Equations 259 (8) (2015)  3929--3954.

\bibitem{vz} J.--L. Vazquez, E. Zuazua, {\em The Hardy inequality and the asymptotic behaviour of the heat equation with an inverse--square potential,} J. Funct. Anal. { 173} (2000), 103--153.


\end{thebibliography}
\end{document}